\documentclass[12pt]{article}

\usepackage{amssymb}
\usepackage{amsmath}
\usepackage{amsthm}
\usepackage[T1]{fontenc}
\usepackage[utf8]{inputenc}
\usepackage[english]{babel}
\usepackage[section]{algorithm}
\usepackage{algorithmic}
\usepackage{graphicx} % Pour le \includegraphics
\usepackage[normalsize]{subfigure}
\usepackage{caption}
\usepackage{url}
\usepackage{tikz-cd}
\providecommand{\keywords}[1]
{
  \small	
  \textbf{\textit{Keywords : }} #1
}

\newtheorem{theorem}{Theorem}[section]
\newtheorem{definition}[theorem]{Definition}
\newtheorem{proposition}[theorem]{Proposition}
\newtheorem{lemma}[theorem]{Lemma}
\newtheorem{remark}[theorem]{Remark}
\newtheorem{corollary}[theorem]{Corollary}
\newtheorem{example}[theorem]{Example}
\newtheorem{remarks}[theorem]{Remarks}

\DeclareMathOperator*{\Dom}{Dom\,}
\DeclareMathOperator*{\Ima}{Im\,}

\usepackage{geometry}
\usepackage{float}

\begin{document}

	\title{Dowker Complexes and filtrations on self-relations}
	\author{Dominic Desjardins Côté}
	
	\maketitle
	\bibliographystyle{plain}	
	\begin{abstract}
		Given a relation on $ X \times Y $, we can construct two abstract simplicial complexes called Dowker complexes. The geometric realizations of these simplicial complexes are homotopically equivalent. We show that if two relations are conjugate, then they have homotopically equivalent Dowker complexes. From a self-relation on $ X $, this is a directed graph, and we use the Dowker complexes to study their properties. We show that if two relations are shift equivalent, then, at some power of the relation, their Dowker complexes are homotopically equivalent. Finally, we define a new filtration based on Dowker complexes with different powers of a relation. 
	\end{abstract}		
	
	\setlength{\parskip}{1em}	
	
	\section{Introduction}

		We can use multivalued maps to study dynamical systems \cite{arMvMap}. The idea is to use Conley index \cite{boConleyIndex} on upper semi-continuous multivalued maps. In applications, it can be hard to study a dynamical system. We can use a model that seems to fit data, but it can be a challenge to find it. Another way is to discretize the continuous space and use to multivalued maps to approximate the underlying dynamical system \cite{arAppDataMv} \cite{arPHDSzym}. Another approach is to use combinatorial structures. To name a few, we can use combinatorial vector fields from Forman \cite{RobForCVF} \cite{towFormalTie} \cite{arFinDataVecDom}. Moreover, a generalization was proposed by Mrozek called combinatorial multivector fields \cite{arMulVecConMorFor} \cite{arCombVsClass}. Finally, others proposed to use the distributive lattices to compute attractors on finite data \cite{arLatAtt1} \cite{arLatAtt2} \cite{arLatAtt3}.

		Multivalued maps can be restrictive. In  \cite{towFormalTie}, authors generalize them to partial multivalued maps. But a partial multivalued map is equivalent to a relation. Some advancements were done in \cite{arSzym} by using the Scymczak category of finite sets where objects are sets and morphisms are relations. The Szymczak category \cite{arMainSzym} captures the essence of index pairs and index maps \cite{boConleyIndex} which is the core of the theory of Conley index. So one motivation of this paper is to continue to develop the theory of relations, and it can be used to study dynamical system with finite data.

		Our main object is a relation which is a subset of the cartesian product of two sets $ X $ and $ Y $. We can define two different abstract simplicial complexes on a relation. For the first simplicial complex, we fixed a value $y \in Y$. For all elements in $ X $, that they are related to $y$, they will span a simplex together. For the second one, we reverse the role. We fixed a value $x \in X$. For all elements in $Y$, that they are in relation with $x$, they will span a simplex together. They are called Dowker complexes \cite{arDowkerCmp}. An important result is the Dowker's Theorem. It says that the geometric realization of these Dowker complexes are homotopically equivalent. The Dowker's Theorem is quite useful in applications. To name a few examples, we can use Dowker complexes to study signal coverage \cite{topoLandSignal}, to find errors in relation of programs and files \cite{arDowTopoDifTest}, to study the privacy of information \cite{topoPrivacy} and in social studies \cite{arQanalAppSoc}. For the last one, this method is called Q-analysis which is developed by Atkins \cite{arQanalAtkin}. The idea of Q-analysis is to study the $q$-connectivity and the $q$-tunnel of the Dowker complexes. Our two main inspirations for definitions come from these articles \cite{arGeoRel} and \cite{arSzym}.

	They are many contributions in this article. First, conjugate relations have homotopically equivalent Dowker complex. If two relations are shift equivalent with lag $l$, then at a certain power their Dowker complexes are homotopically equivalent. We can define a filtration on relation based on the Dowker complex at different powers of a finite self-relation. Moreover, this can be computed in finite time. We can also compute the $0$th homology of a high enough power with the connected components of the graph induced by the relation.
		
		 This article goes as follows. In section 2, we remind some concepts and definitions on finite relation, graph, simplicial complex, Dowker complexes and finally the famous Dowker's Theorem. In section 3, we define right morphism and left morphism. If there exists a right or left morphism  between two relations, then there is an inclusion from one of the Dowker complexes. Moreover, we show that if there exists a conjugacy between two relations, then their Dowker complexes are homotopically equivalent. In section 4, we generalize the idea of right and left morphism for multivalued right and multivalued left morphism. We show two important properties needed for the definition of a filtration. The Dowker complex of a certain power of a relation is include in the Dowker complex of the same relation with a higher power. For some finite relations at certain a power $ j $, every other Dowker complexes of the same relation at power higher than $j$ are the same. We call it the stabilization of the Dowker complex. We show that shift equivalence between relations have homotopically equivalent Dowker complex at some power. In section 5, we can define a filtration on the Dowker complexes of different powers of a relation under some simple conditions by using the two properties in section 4. We can use persistent homology on these filtration to extract topological features of the Dowker complex. If a relation is acyclic, then we have that the number of connected components of the graph associated to the relation up to a certain power is equal to the dimension of the $0$th homology. It can be generalized to the class of simple relations. We also have a similar result for strongly connected relations.
		
	\section{Preliminaries}
	
		\subsection{Finite Relations}	
		
		Let $ X $ and $ Y $ be finite sets. We define a relation as a subset of $ X \times Y $. Let $ (x,y) \in R \subset X \times Y $, we denote by $ x R y $ or by $ y \in R(x) $. We define the composition of relations as follows. Let $ R_1 \subset X \times Y $ and $ R_2 \subset Y \times Z $. 
		
		\begin{equation}
		R_2 \circ R_1 := \{ (x,z) \in R_2 \circ R_1 \mid \exists y \text{ such that } x R_1 y \text{ and } y R_2 z \}.
		\end{equation}
		
		We define the inverse relation by swapping the sets of a relation.
		
		\begin{equation}
			R^{-1} := \{ (y,x) \in Y \times X \mid y \in R(x) \}
		\end{equation}
	
		If a relation is a subset of $ X \times X $, then we say it's a self-relation on $X$. We define the power of a self-relation as follows : 
		
		\begin{equation*}
			R^n := 	\begin{cases}
				R \circ R^{n-1} & n > 0 \\
				Id_X & n = 0 \\
				R^{-1} \circ R^{n+1} & n < 0
			\end{cases}
		\end{equation*}
	
	The domain and the image for a relation $ R \subset X \times Y $ are :
	
	\begin{gather}
		\Dom R := \{ x \in X \mid \exists y \text{ such that } (x, y) \in R \} \\
		\Ima R := \{ y \in Y \mid \exists x \text{ such that } (x, y) \in R \}.
	\end{gather}
	
	We can see relations as partial multivalued maps. If $ \Dom R = X $, then we say that the relation is a multivalued map. A relation is injective, if for all $x_1, x_2 \in X $, $R(x_1) = R(x_2) $ implies that $ x_1 = x_2 $. A relation is surjective if $ \Ima R = Y $. Moreover, a map $ f : X \to Y $ induces a relation where $(x, f(x)) \in R$. Without ambiguity, we can compose maps and relations together to obtain a new relation.
	
	\begin{definition}	\label{dfnEventPeriod}
		Let $ R $ be a self-relation on $X$. Let $j$ be the least positive integer such that :
		\begin{equation*}
			R^{j} = R^{j + p} \text{ for some } p > 0.
		\end{equation*}
		We say that $ j $ is the index and the least $ p > 0 $ is the period. If $j = 1$, then $R$ is periodic. A pair $(j, p)$ is the eventual period of $R$ with index $j$ and period $p$.
	\end{definition}	

	In other words, the period $p$ will eventually be a period for $R$.

	We sometimes use matrices with values in $ \lbrace 0, 1 \rbrace $ to represent relations. Let $ R \subset X \times Y $ be a relation with $ card(X)  = m $ and $ card(Y) = n $. The matrix $ M_{m \times n} $ have a value $ 1 $ at $ M_{i, j} $ if $ x_i R y_j $ otherwise the value is $0$. It can be called relation matrix, Boolean relation matrix, binary relation matrix, binary Boolean matrix, $(0, 1)$-Boolean matrix and $(0 ,1)$-matrix. For more information on Boolean matrix theory, we refer to the book \cite{boolMatTheory}.
		
	We say that a self-relation $R$ on $X$ has a cycle at $ x $ if and only if there exists an $ n \in \mathbb{N} $ such that $ x R^n x $. We say $R$ as a fixed point at $ x $, if $ n = 1 $. If a relation has no cycle at $x$ for all $ x $ with period $ n > 1 $, then the relation is acyclic. A cycle is a sequence $ x_1, x_2, \ldots, x_n $ such that $ x_1 = x_n$ and $ x_i R x_{i + 1} $. A self-relation $ R $ on $ X $ is simple if for any two cycles are either disjoints or equals.	
 	
	\subsection{Graphs}
	
		In this subsection, we remind the definition of a graph and some notations. 
		
		\begin{definition}
			A directed graph $ G $ is a pair $(E, V) $ where $ V$ is the set of vertices $ V $ and $ E$ is the subset $V \times V $ the set of edges.	
		\end{definition}		
		
		A relation can also be seen as a directed graph. If $ R $ is a self-relation on $X$, then $X$ is the set of vertices and the set of edges $ E = R $. This graph has at most one directed edge from the vertex $A$ to the vertex $B$, and we also allow a self-loop on vertices. We note $ G_R $ the graph induced by a self-relation $R$.
		
		Let $ x, y \in V $. There is a $(x, y)$-path, if there exists a sequence of edges $ e_1, e_2, \ldots, e_n \in E $ that connect $ x $ to $y$ without following the direction of edges. We can define an equivalence relation on vertices of $G$. If there is a path between two vertices $ x $ and $ y $, then $ x $ and $ y $ are in the same class of equivalence. For a graph $ G $, we say the number of connected components is the number of class equivalences of the relation of paths. We say $ G $ is connected if there is only one connected component.
		
		 If the sequence of edges of a $(x, y)$-path follows the direction of edges of the graph, then we say it's a $(x, y)$-walk. We can also define an equivalence relation with a walk between vertices. If there is a walk from $ x $ to $ y $ and a walk from $ y $ to $ x $, then $x $ and $ y $ are in the same equivalence class. This is the class of strongly connected components. For a graph $ G $, we say the number of strongly connected components is the number of class equivalence of the relation of walks. We say $ G $ is strongly connected if there is only one strongly connected component. 
					
	\subsection{Simplicial Complexes and Dowker Complexes}
	
		In this subsection, we will discuss simplicial complex, Dowker complex and the Dowker's Theorem. For more information about simplicial complex, we suggest to read \cite{AlgTopo}. We do not present filtration and persistent homology, but we refer to \cite{boCompTopo}.
		
		An abstract simplicial complex is a set $K$ that contains finite non-empty sets such as  if $ A \in K $, then for all subsets of $A$ are also in $K$. For further examples, we use geometric simplex. A geometric $n$-simplex is the convex hull of a geometrically independent sets of vertices $ \lbrace v_0, v_1, \ldots, v_n \rbrace \in \mathbb{R}^N $. This is the set of $ x \in \mathbb{R}^N $ such as $ x = \sum_{i=0}^n t_i x_i $ and $ 1 = \sum_{i=0}^n t_i $ where $ t_i \geq 0 $ for all $i$. We denote an $n$-simplex by $ [ v_0, v_1, \ldots, v_n ] $ is the simplex spanned by the vertices $ v_0, v_1, \ldots, v_n $. Any simplex spanned by the subsets of $ \lbrace v_0, v_1, \ldots, v_n \rbrace $ are called faces and denote by the symbol $ \leq $. A simplicial complex is a collection of simplices for all $ \sigma \in K $, if $ \tau \leq \sigma $ then $ \tau \in K $ and if $ \sigma_1 \cap \sigma_2 = \tau $, then $ \tau $ is either the empty set or $ \tau $ is a face of $ \sigma_1 $ and $ \sigma_2 $. We say that $ L \leq K $ if $L$ is a sub-complex of $K$. A simplicial complex is contractible if its homology is equivalent to a point. Given an abstract simplicial $K$, we can define a geometric simplicial complex and $|K|$ call the geometric realization of $K$. We call $0$-simplices vertices and $1$-simplices edges. The closure of a simplex $ \sigma $ is the set of all the faces of the simplex. We denote it by $ cl(\sigma)$. We need one more definition related to simplicial complexes. It will be useful in some proofs.
		
		\begin{definition}
			A simplicial complex $K$ is edge-connected, if for any two vertices $ x $ and $ y $ there is a sequence of edges $ e_1, e_2, \ldots, e_n $ such that $ x \in e_1 $, $ y \in e_n $ and $ cl(e_i) \cap cl(e_{i+1}) \neq \emptyset $ for all $ i = 1, 2, \ldots, n-1 $.
		\end{definition}
		
		We have that the simplicial complex is connected if and only if it is edge-connected if and only if $ H_0 $ is dimension $1$ \cite{boCompHomo}.	
		
		Now we explain how to construct abstract simplicial complexes from a relation which are called Dowker complexes. Let $ R \subset X \times Y $ be a relation and $ X, Y $ be two finite sets. There are two ways to construct the Dowker complex from a relation.
	
	\begin{definition} \label{dfnDowHor}
		Let $ R \subset X \times Y $ be a finite relation and $ K_R $ be the Dowker complex. A simplex $ [x_1, x_2, \ldots, x_n] \in K_R $ if and only if $ \exists y \in Y $ such as $ x_i R y $ for all $ i =1,2, \ldots, n $.
	\end{definition}
	
	We have an analogous construction.	
	
	\begin{definition} \label{dfnDowkerCmp}
		Let $ R \subset X \times Y $ be a finite relation and $ L_R $ be the Dowker complex. A simplex $ [y_1, y_2, \ldots, y_m] \in L_R $ if and only if $ \exists x \in X $ such as $ x R y_i $ for all $ i = 1, 2, \ldots, n $.
	\end{definition}
	
	We denote $ [ x_1, x_2, \ldots x_n ] = \sigma_y \in K_R $ if and only if $ x_i R y $ for all $ i = 1, 2, \ldots, n $. We use $ y $ as an index for $\sigma_y $ to note that all vertices of $ \sigma_y $ are in $ R^{-1}(y) $. We use the same notation for $ \sigma_x \in L_R $ but the vertices are in $ R(x)$.	
	
	By using the matrix notation, we can use rows and columns to build the simplices. The columns are for $K_R$ and the rows are for $ L_R $.
	
	\begin{example}	\label{exDwkCMp}
		Let $ R \subset X \times Y $  be a finite relation.
		
		\begin{equation}
			R := \begin{bmatrix}
			1 & 0 & 0 & 0 & 1 \\
			0 & 0 & 1 & 1 & 0 \\
			1 & 0 & 0 & 0 & 1 \\
			1 & 1 & 0 & 0 & 0
			\end{bmatrix}
		\end{equation}		 
		
		The first column gives the $2$-simplex $ [x_1, x_3, x_4]  $. The third and the fourth column give the $0$-simplex $ [x_2] $. The second and the fifth column do not add new simplices. We obtain the simplicial complex $ K_R := \lbrace [x_1, x_3, x_4], [x_2] \rbrace $.

		The first row adds a $1$-simplex $ [y_1, y_5] $ to $L_R$. The second row gives a $1$-simplex $ [y_3, y_4] $. The final row adds a $1$-simplex $[ y_1, y_5 ]$. We obtain the simplicial complex $L_R = \lbrace [y_1, y_5], [ y_3, y_4 ], [y_1, y_2] \rbrace $.

		We obtain that $|K_R|$ and $|L_R|$ have two connected components and no higher dimension cycle.
		
		\begin{figure}
  			\center
  			\subfigure[ Geometric realization of the Dowker complex $ K_R $. ]{
   				\includegraphics[height=6cm, width=5cm, scale=1.00, angle=0 ]{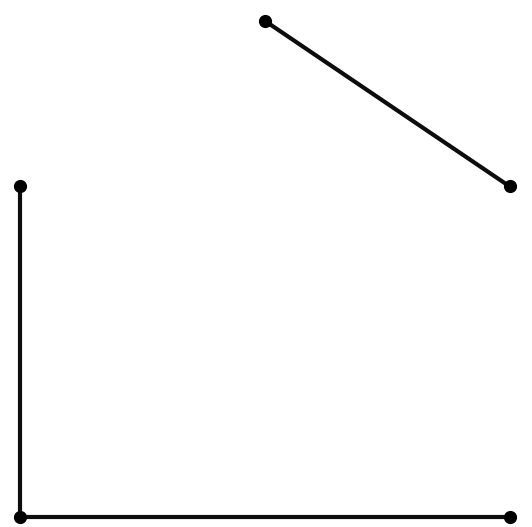}
  			}
  			\,
  			\subfigure[ Geometric realization of the Dowker complex $L_R$. ]{
   				\includegraphics[height=5cm, width=5cm, scale=1.00, angle=0 ]{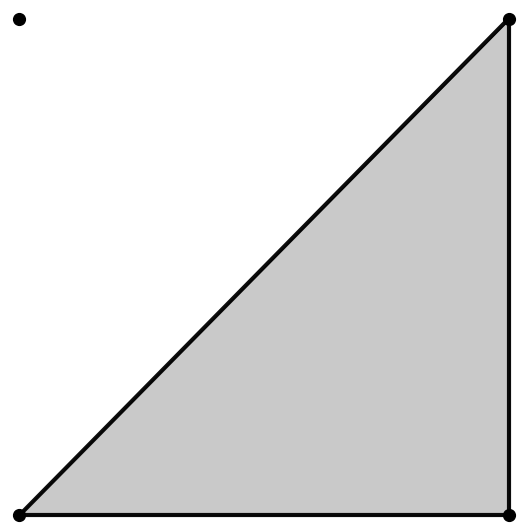}
            }
            \caption{These are geometric realizations of the Dowker complexe in Example \ref{exDwkCMp}. They are homotopically equivalent. }
        \end{figure}

	\end{example}	
	
	The next theorem links to the homotopy between $ |K_R| $ and $ |L_R| $. 
	
	\begin{theorem}[Dowker's Theorem] \label{thmDowker}
		Let $R \subset X \times Y$ be a relation and let $K_R$ and $L_R$ be the associated Dowker complexes. Then, the polyhedra $|K_R|$ and $|L_R|$ are homotopy equivalent.
	\end{theorem}

		In 1952, Dowker \cite{arDowkerCmp} has shown that $ K_R $ and $ L_R$ have isomorphic homology groups. In 1995, Björner \cite{arBjorner} has shown that $ | K_R | $ and $ | L_R | $ are homotopically equivalent, which is the more commonly use in the literature. In recent years, Dowker complexes were regained in popularity in the community of topology data analysis. We can use them to do a filtration on weighted networks \cite{arDowThPersHomo}. In our cases, our filtration will be different and based on different powers of a self-relation.
	
	\section{Left and Right Morphism}
	
		Let start with the definition of the graph homomorphism and next we define left and right morphisms between relations.
		
		\begin{definition}
			Let $ R $ be a self-relation on $X$ and $ R' $ be a self-relation on $Y$. A map $f : X \to Y $ is a graph homomorphism if the following condition is satisfied :
			
			\begin{equation*}
				\text{For every }  x_1, x_2 \in X \text{ such as } x_1 R x_2 \implies f(x_1) R' f(x_2).
			\end{equation*}
			
			If $ f $ is bijective and its inverse is also a graph homomorphism, then $ f $ is a graph isomorphism.			
		\end{definition}		
		
		We obtain that graph homomorphism keeps some information of the Dowker complex coming from the first relation.		
		
		\begin{lemma}\label{lemGraphHomoInj}
			Let $ f : X \to Y $ be a graph homomorphism between $ R $ and $ R' $. If $ f $ is injective, then there exist a map $ p: K_R \hookrightarrow K_{R'} $.
		\end{lemma}
		\begin{proof}
			Consider a $n$-simplex $ [ x_0, x_1, \ldots, x_n ] \in K_R $. Then, there exists $ \alpha \in X $ such as $ x_i R \alpha $ for all $ i = 0,1,2, \ldots n $. We have that $ f $ is a graph homomorphism. This implies that $ f(x_i) R' f(\alpha) $ for all $ i =0, 1, 2, \ldots, n $. Indeed, $f$ is injective implies that $ [f(x_0), f(x_1), \ldots, f(x_n)] $ is also a $n$-simplex in $ K_{R'} $. So we can construct a map $p : K_R \hookrightarrow K_{R'} $ by sending a simplex $ [x_0, x_1, \ldots, x_n] $ to $ [f(x_0), f(x_1), \ldots, f(x_n)] $. By the previous argument, $ p $ is well defined and injective.
		\end{proof}	 

		If we have a graph isomorphism between two relations, then the Dowker complexes remain unchanged. This holds because graph isomorphisms are relabelling on the vertices of a graph.

		\begin{proposition}
			Let $ R_1$ be a self-relation on $X$ and $ R_2 $ be a self-relation on $Y$. If there exists a graph isomorphism $f$ between $ R_1 $ and $ R_2 $, then they have the same Dowker complexes up to the label of vertices. 
 		\end{proposition}
		\begin{proof}
			Graph homomorphisms $ f $ and $ f^{-1} $ are injective. By the Lemma \ref{lemGraphHomoInj}, there exist two injective maps $ p: K_R \hookrightarrow K_{R'} $ and $ p' : K_{R'} \hookrightarrow K_R  $. So we have that $ K_R $ and $ K_{R'} $ are the same up to the label of vertices.
		
		By similar arguments, we can show it for $ L_R $ and $ L_{R'} $.
		\end{proof}
					
		Graph homomorphisms are nice, but they can drastically change the Dowker complexes. So, we defined a left morphism which it changes the source of an edge and a right morphism which it changes the target of an edge. In this way, only one of the Dowker complexes will change from the right morphism or the left morphism.
		
	\begin{definition}\label{dfnRigMor}
		A right morphism $ f:(X, Y, R) \to (X, Z, R') $ is a map $ f : Y \to Z $ such that for every $ x \in X $ and $y \in Y $ :
		\begin{equation*}
			x R y \implies x R' f(y).
		\end{equation*}
	\end{definition} 	
	
	We obtain this simple Lemma which is very useful for later proofs.	
	
	\begin{lemma}\label{lemRigMor}
		If there exists a right morphism $ f:(X, Y, R) \to (X, Z, R') $, then $ K_R \leq K_{R'} $. We obtain the equality if $ f $ is a bijective map.
	\end{lemma}
	\begin{proof}
		Let $ f :(X, Y, R) \to (X, Z, R') $ be a right morphism and $ [x_1, x_2, \ldots, x_n] \in K_R $. This implies there exists a $ y \in Y$ such as $ x_i R y $ for all $ i = 1, 2, \ldots, n $. We obtain that $ x_i R' f(y) $ for all $ i  $. Finally, we have $ [x_1, x_2, \ldots, x_n] \in K_{R'} $.
		
		Now, we suppose that $f$ is bijective. Let $ [ x_1, x_2, \ldots, x_n ] \in K_{R'} $. Then, there exists a $ z \in Z $ such that $ x_i R' z $ for all $ i = 1, 2, \ldots, n $. We have $ f^{-1}(z) \in Y  $ and $ f^{-1} $ is well defined because $ f $ is bijective. Then, $ x_i R f^{-1}(z)$ for all $ i = 1, 2, \ldots, n $. We obtain that $ [ x_1, x_2, \ldots, x_n ] \in K_R $. 
	\end{proof}
			
		The idea of right morphism comes from the article \cite{arGeoRel}. The author only considered the right morphism. But, in our case, we are also interesting of modifying the first set in the cartesian product of a relation.			
		
	\begin{definition}\label{dfnLeftMor}
		A left morphism $ g : (X, Z, R) \to (Y, Z, R') $ is a map $ g : X \to Y $ such that for every $ x \in X $ and $ z \in Z $ :
				
		\begin{equation*}
			x R z \implies g(x) R' z.
		\end{equation*}
	\end{definition}			
				
		We have an analogous Lemma for left morphism as the Lemma \ref{lemRigMor} for right morphism.	
		
	\begin{lemma} \label{lemLeftMor}
		If there exists a left morphism $ g : (X, Z, R) \to (Y, Z, R) $, then $ L_R \leq L_{R'} $. We obtain the equality if $ g $ is a bijective map.
	\end{lemma}
	
	With the definition of right and left morphism, we can easily show that if two relations are conjugate, then there Dowker complexes are homotopically equivalent. We remind the definition of conjugacy between relations before showing the proof.	
	
	\begin{definition}
		Let $ R_1 $ be a self-relation on $X$ and $ R_2$ be a self-relation on $Y$. We say that $ R_1 $ and $ R_2 $ are conjugate if there exists a bijective map $ \varphi : X \to Y $ such as $ \varphi \circ R_1 = R_2 \circ \varphi $.
	\end{definition}	
	
	\begin{corollary}
		Let $ R  $ be a self-relation on $X$ and $ R'$ be a self-relation on $Y$ which are conjugate by a bijective map $ \varphi : X \to Y $. Then, $ |K_R|, |L_R|, |L_{R'}| $ and $ |K_{R'}| $ are homotopy equivalent.
	\end{corollary}		
	\begin{proof}
		The map $ \varphi $ is bijective. It implies that $ K_R = K_{\varphi \circ R} $ by Lemma \ref{lemRigMor} and $ L_{R'} = L_{R' \circ \varphi} $ by Lemma \ref{lemLeftMor}.
		
		By Dowker's Theorem, we obtain that $ |K_R| $ is homotopic equivalent to $ |L_{R'}| $, because $ K_{\varphi \circ R} = K_{R' \circ \varphi} $.	
	\end{proof} 	
	
	In \cite{arCoshefDow}, the author decides to combine the right and left morphism together. Let $ R \subset X \times Y $ and $ R' \subset X' \times Y' $ be relations and $ f : X \to X' $ and $ g : Y \to Y' $ be two maps. A pair $ (f, g) $ is a morphism between relation $ R_1 $ and $ R_2 $ if for all $ x \in X, y \in Y  $ such that $ x R_1 y $ it implies that $ f(x) R_2 g(y) $. In \cite{arCoshefDow}, it is shown that the Dowker complex and (co)sheaf representation have nice functoriality properties. In our case, it won't be useful because we only need right or left morphism. But we can see them as a pair $ (id_X, f) $ where $ f $ is a right morphism and $ id_X $ is the identity function on $X$. 
	
	\section{Multi-right morphism and multi-left morphism}	
	
	We want to work with multivalued maps. We generalize left and right morphism to multi-left and multi-right morphism.
	
	\begin{definition}
		A multi-right morphism $ F : (X, Y, R) \multimap (X, Z,R')$ is a multivalued map $ F : Y \multimap Z $ such as for all $ x \in X $, $ y \in Y $ :
		\begin{equation*}
			x R y \implies x R' a \text{ for all } a \in F(y).
		\end{equation*}
	\end{definition}
	
		We also obtain the same Lemma as before.
		
	\begin{lemma}
		Let $ R \subset X \times Y $ and $ R' \subset X \times Z $ be relations. If there exist a multi-right morphism $ F : (X, Y, R) \multimap (X, Z, R') $, then $ K_R \leq K_{R'} $. We obtain the equality if $ F $ is a bijective multivalued map.
	\end{lemma}
	\begin{proof}
		The proof is the same as Lemma \ref{lemRigMor}.
	\end{proof}	
	
	\begin{definition}
		A multi-left morphism $G : (X, Z, R) \multimap (Y, Z, R') $ is a multivalued map $ G : X \to Y $ such for all $ x \in X, z \in Z $ :
				
		\begin{equation*}
			x R z \implies a R' z \text{ for all } a \in G(x).
		\end{equation*}
	\end{definition}	
	
	\begin{lemma}
		Let $ R \subset X \times Z $ and $ R' \subset Y \times Z $ be relations. If there exist a multi-left morphism $ G : (X, Z, R) \multimap (Y, Z, R') $, then $ L_R \leq L_{R'} $.  We obtain the equality if $ G $ is a bijective multivalued map.
	\end{lemma}	
	
	We denote a multi-right morphism by mr-morphism and a multi-left morphism by ml-morphism.	
	
	\begin{remarks}
		We remind that if a relation $ S \subset X \times Y $ satisfies $ \Dom S = X $, then $ S $ is a well-defined multivalued map. Moreover, for any relation $ R \subset Z \times X $, we have that $ S : (Z, X, R) \multimap (Z, Y, S \circ R) $ is a well-defined mr-morphism. It is also true for ml-morphism.
	\end{remarks}		

	The next corollary will be useful to define our filtrations.

	\begin{corollary}\label{lblCorRelPower}
			Let $ R $ be a self-relation on $X$. If $ \Dom R = X $, then $ K_{R^n} \leq K_{R^{n+1}} $. If $ \Ima R = X $, then $ L_{R^n} \leq L_{R^{n+1}} $.
	\end{corollary}
	\begin{proof}	
	$ \Dom R = X $ implies that $R$ is a multivalued map. Moreover, $ R : (X, X, R^n) \multimap (X, X, R^{n+1}) $ is well-defined mr-morphism. By Lemma \ref{lemRigMor}, we have that $ K_{R^n} \leq K_{R^{n+1}} $.
	
	In the same way, $ \Ima R = X $ implies that $R^{-1} $ is a multivalued map. We have that $ R^{-1} : (X, X, R^n) \multimap (X, X, R^{n+1}) $ is well-defined ml-morphism. By Lemma \ref{lemLeftMor}, we have that $ L_{R^n} \leq L_{R^{n+1}} $.
	\end{proof}
	
	\begin{remark}	
		The hypothesis $ \Dom R = X $ in the previous corollary is important. Let us show that by an example. If $ R$ is a self-relation on $X$, define by this matrix such as :
			\begin{equation}
				\begin{bmatrix}
				0 & 1 & 1 \\
				0 & 0 & 1 \\
				0 & 0 & 0
				\end{bmatrix}
			\end{equation}
		Then, we obtain $ K_{R} = \{ [ x_2, x_3 ], [x_2], [x_3] \} $, $ K_{R^2} = \{ [ x_3 ] \} $ and $ K_{R^n} = \emptyset $ for $ n > 2 $.
		
		We don't have $ K_{R^n} \subset K_{R^{n+1}} $ for all $ n \in \mathbb{N}_{>0} $. If the matrix is nilpotent, then $ K_{R^n} $ is an empty set for an $ n \in \mathbb{N}_{>0} $.	
	\end{remark}

	Using Corollary \ref{lblCorRelPower}, we can show that the Dowker complexes of a relation stabilize at some power $n$.	
	
	\begin{corollary} \label{corEvePerRel}
			Let $ R $ be a finite self-relation on $X$ with an eventual period $(j, p)$. If $ \Dom R = X $, then, we have $ K_{R^j} = K_{R^{j+i}} $ for $ i \in \mathbb{N} $. If $\Ima R = X $, then $ L_{R^j} = L_{R^{j+i}} $ for $ i \in \mathbb{N} $. 
		\end{corollary}		
		\begin{proof}
			By Corollary \ref{lblCorRelPower}, we have the sequence :
			\begin{equation}\label{eqComp}
				K_{R^j} \leq K_{R^{j+1}} \leq \ldots \leq K_{R^{j + p - 1}} \leq K_{R^{j + p}}.
			\end{equation}
			
			But $p $ is the period of $R$, hence $ R^j = R^{j + p} $ implies  that $ K_{R^j} = K_{R^{j+p}} $. By (\ref{eqComp}), we obtain $ K_{R^j} = K_{R^{j+i}} $  for $ i \in \mathbb{N} $. A similar proof can be done for  $L_{R^j} = L_{R^{j+i}} $ for $ i \in \mathbb{N} $.
		\end{proof}		
		
		We remind the definition of shift equivalence between two relations and we show the assumption that relations are shift equivalences implies that theirs Dowker complexes are homotopically equivalent at some power for each relation.
		
		\begin{definition}
			Let $ R_1 $ be finite self-relation on $X$ and $ R_2 $ be finite self-relation on $Y$. $ R_1 $ and $ R_2 $ are shift equivalent with a lag $ l $, if there exists two relations $ S \subset X \times Y $ and $ T \subset Y \times X $ such as :
			\begin{gather*}
				R_1 \circ T = T \circ R_2 \quad S \circ R_1 = R_2 \circ S \\
				T \circ S = R_1^l \quad  S \circ T = R_2^l	
			\end{gather*}	
			
			We say it is a strong shift equivalence if $ l = 1$. 		
		\end{definition}		
		
		\begin{corollary}\label{corShiEqui}
			Let $ R_1 $ be finite self-relation on $X$ with $ \Dom R_1 = X = \Ima R_1 $ and $ R_2 $ be finite self-relation on $Y$ with $ \Dom R_2 = Y = \Ima R_2 $. Let $(j_p, p)$ be an eventual period of $R_1$ and $(j_q, q) $ be an eventual period of $ R_2 $. Without loss of generality, we suppose that $ j_p \geq j_q $. If $ R_1 $ and $ R_2 $ are shift equivalent with lag  $ l $, then  $| K_{R_1^{j_p}}| $, $ | K_{R_2^{j_q}} |$, $| L_{R_1^{j_p}}|$ and $ | L_{R_2^{j_q}} |$ are homotopy equivalent. 
		\end{corollary}	
		\begin{proof}	
 	If $ R_1 \circ T = T \circ R_2 $ then $ R_1^n \circ T = T \circ R_2^n $ is also true for $ n \in \mathbb{N} $. Moreover, we have $ \Dom S = X $ and $ \Dom T = Y $, because $ \Dom R_1 = X $, $ \Dom R_2 = Y $, $T \circ S = R_1^l$ and $  S \circ T = R_2^l $. So we have that $ S $ and $ R $ are well defined multivalued maps.
		
			We want to show that $ K_{R_1^{j_p}} = K_{S \circ R_1^{j_p}} $ and $ L_{R_2^{j_q}} = L_{R_2^{j_q} \circ S} $. We are going to use $ T $ and $ S $ as mr-morphism with $ R_1$ and ml-morphism with $R_2 $.

			We have that $ S : (X, X, R_1^{j_p}) \multimap (X, Y, S \circ R_1^{j_p}) $ and $ T : (X, Y, S \circ R_1^{j_p}) \multimap (X, X, T \circ S \circ R_1^{j_p}) $ are well-defined mr-morphisms. We have $ T \circ S = R_1^l $. It implies that $ T \circ S \circ R_1^{j_p} = R_1^{l+j_p} $. We obtain $ K_{R_1^{j_p}} \leq K_{ S \circ R_1^{j_p}} \leq K_{T \circ S \circ R_1^{j_p}} = K_{R_1^{l+j_p}} $. The eventual period of $R_1$ is $(j_p, p)$. It implies that $ K_{R_1^{j_p}} = K_{R_1^{j_p+l}} $. We obtain that $K_{R_1^{j_p}} = K_{ S \circ R_1^{j_p}} $.

			We can see that $S : (Y, Y, R_2^{j_p}) \multimap (X, Y, R_2^{j_p} \circ S) $ and $ T : (X, Y, R_2^{j_p} \circ S) \multimap (Y, Y, R_2^{j_p} \circ S \circ T) $ are well-defined ml-morphisms. We have $ S \circ T = R_2^l $. It implies that $ R_2^{j_p} \circ S \circ T = R_2^{j_p+l} $. We obtain $ L_{R_2^{j_p}} \leq L_{R_2^{j_p} \circ S} \leq L_{R_2^{j_p} \circ S \circ T} = L_{R_2^{j_p+l}} $. We have $ L_{R_2^{j_q}} = L_{R_2^{j_p+l}} $ because $ j_p \geq j_q $. We obtain $ L_{R_2^{j_p}} = L_{R_2^{j_p} \circ S} $.
			
		Finally, we have $ K_{R_1^{j_p}} = K_{ S \circ R_1^{j_p}} = K_{R_2^{j_p} \circ S} $ and $ L_{R_2^{j_p}} = L_{R_2^{j_p} \circ S} = L_{S \circ R_1^{j_p}} $. By Dowker's Theorem, we obtain that $| K_{R_1^{j_p}} |$, $ |K_{R_2^{j_q}}| $, $| L_{R_1^{j_p}}|$ and $ |L_{R_2^{j_q}}| $ are homotopy equivalent. 
				
		\end{proof}				
				
		\begin{example} \label{exShift}
			
			Let $ X $ be a finite set with $ 8 $ points and $ Y $ be a finite set with $3$ points. Let $ R_1$ be a self-relation on $X$ and $ R_2 $ be a self-relation on $X$ defined by those graphs in Figure \ref{figExGraphShift}. $ R_1 $ has an eventual period $(3, 3)$ and $ R_2 $ has an eventual period $(1, 3)$. We see in Figures \ref{figExDowCmpShift}(a), (b) and (d) that the Dowker complexes are not homotopically equivalent. But, in Figures \ref{figExDowCmpShift} (c) and (d), the Dowker complexes with relations at power $3$ are homotopically equivalent. 
			
    	\begin{figure}
  			\center
  			\subfigure[ The graph $ G_{R_1} $. ]{
   				\includegraphics[height=6cm, width=5cm, scale=1.00, angle=0 ]{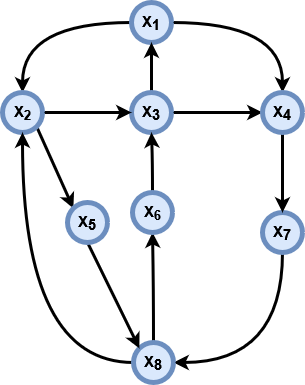}
  			}
  			\,
  			\subfigure[ The graph $ G_{R_2} $.]{
   				\includegraphics[height=5cm, width=5cm, scale=1.00, angle=0 ]{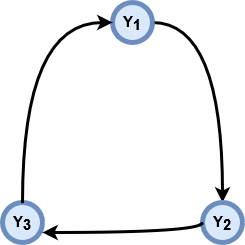}
            }
            \caption{Theses are the graphs from Example \ref{exShift}.}
            \label{figExGraphShift}
         \end{figure}		
						
		\begin{figure}
  			\center
  			\subfigure[ Dowker complex of $ |K_{R_1}| $. ]{
   				\includegraphics[height=6cm, width=5cm, scale=1.00, angle=0 ]{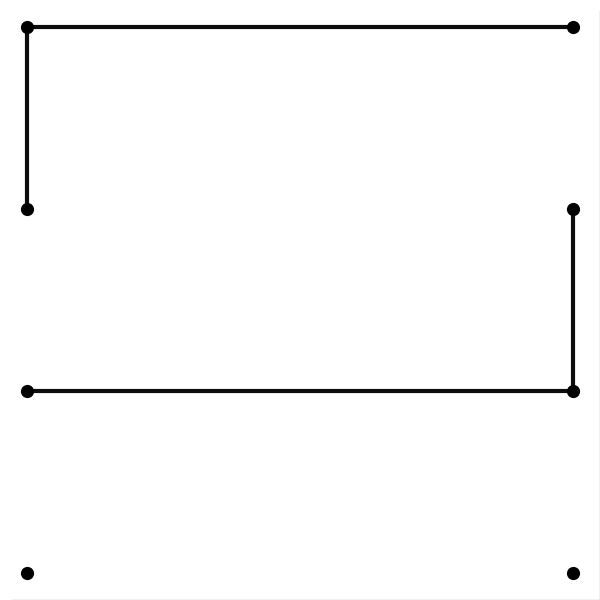}
  			}
  			\quad
  			\subfigure[ Dowker complex of $ |L_{R_1}| $. ]{
   				\includegraphics[height=6cm, width=5cm, scale=1.00, angle=0 ]{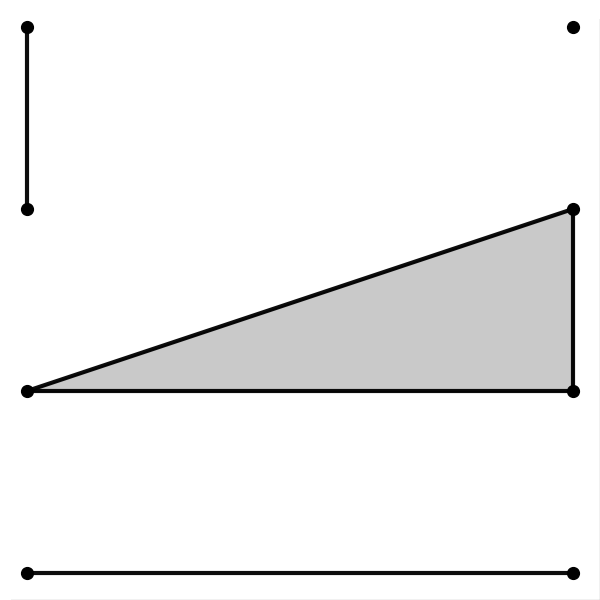}
  			}
  			\quad
  			\subfigure[ Dowker complexes of $ |K_{R_1^3}| $ and $ |L_{R_1^3}| $. ]{
   				\includegraphics[height=5cm, width=6cm, scale=1.00, angle=0 ]{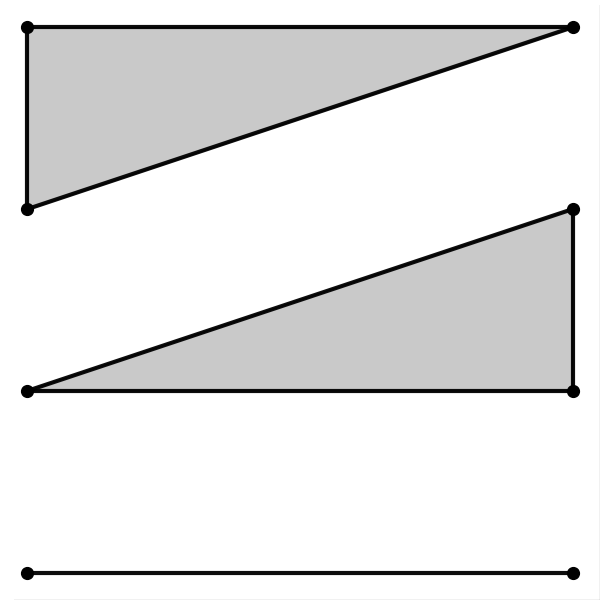}
            }
            \quad
			\subfigure[ Dowker complexes $ |K_{R_2}|, |L_{R_2}|, |K_{R_2^3}| $ and $ |L_{R_2^3}| $. ]{
   				\includegraphics[height=5cm, width=6cm, scale=1.00, angle=0 ]{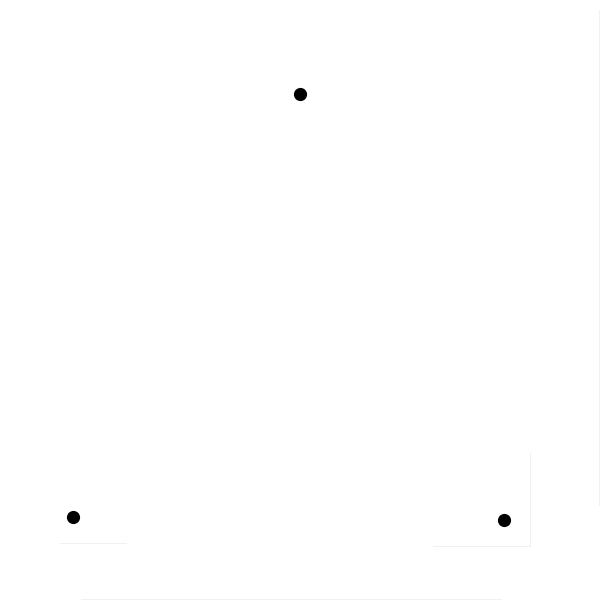}
            }
            \caption{Theses are the different Dowker complexes from Example \ref{exShift}.}\label{figExDowCmpShift}
         \end{figure}		
		\end{example}

		There is an interesting proposition from \cite{arShiftBool} that we can use for strongly connected relations. We remind that an indecomposable Boolean matrix is a relation which is strongly connected and $J$ is a square matrix where all the entries are equals to $1$.
		
		\begin{proposition}[Proposition 4.3 in \cite{arShiftBool} ]	\label{propBoolMatJ}
			Every indecomposable Boolean matrix with positive trace is strong shift equivalent to $J$.
		\end{proposition} 	
		
		We can easily compute the Dowker complexes of $J$. It is a $(n-1)$-simplex where $n$ is the number of rows of $J$ for both Dowker complexes. Finally, we obtain that the Dowker complexes of a strongly connected self-relation at a power high enough are contractible if the trace is strictly positive.

	\section{Filtrations of Dowker complexes}

		For this section, we suppose that a $ R $ is a finite self-relation on $X$. From Corollary \ref{lblCorRelPower}, if $ \Dom R = X $, then we have $ K_{R^i} \leq K_{R^{i+1}} $ for all $ i \geq 0 $. We have an inclusion and we get this filtration :
		
		\begin{equation} \label{eqFilK}
			K_{R} \hookrightarrow K_{R^2} \hookrightarrow \ldots \hookrightarrow K_{R^i} \hookrightarrow K_{R^{i+1}} \hookrightarrow  \ldots
		\end{equation}
		
		In the same way with $ \Ima R = X $, we have $ L_{R^{i}} \leq L_{R^{i+1}} $. We obtain another filtration :
		
		\begin{equation}\label{eqFilL}
			L_{R} \hookrightarrow L_{R^2} \hookrightarrow \ldots \hookrightarrow L_{R^i} \hookrightarrow L_{R^{i+1}} \hookrightarrow  \ldots		
		\end{equation}				
		
		From Corollary \ref{corEvePerRel}, the Dowker complexes stabilize at a certain power. This means we can compute the filtration (\ref{eqFilK}) and (\ref{eqFilL}) in finite time. For our filtrations, we start at $ i = 1 $, but we could also start with the $ i = 0 $. We have that $ R^0 = id_X $ and  $ R : (X, X, id_X) \multimap (X, X, R) $ is a well-defined mr-morphism, if $ \Dom R = X $. Then, the filtration (\ref{eqFilK}) becomes :
		
		\begin{equation}
			K_{id_X} = K_{R^0} \hookrightarrow K_{R} \hookrightarrow K_{R^2} \hookrightarrow	\ldots \hookrightarrow K_{R^i} \hookrightarrow K_{R^{i+1}} \hookrightarrow  \ldots
		\end{equation}
		
		And for the filtration (\ref{eqFilL}) by applying similar arguments, we obtain : 
		
		\begin{equation}
			L_{id_X} = L_{R^0} \hookrightarrow L_{R} \hookrightarrow L_{R^2} \hookrightarrow	\ldots \hookrightarrow L_{R^i} \hookrightarrow L_{R^{i+1}} \hookrightarrow  \ldots	
		\end{equation}
		
		The homology of $ K_{id_X} $ and $ L_{id_x} $ is the homology of $n = card(X)$ points. In some cases, we might want to start the filtration at $ i = 0 $ or $ i = 1 $. 
		
		We remind that, by Dowker's Theorem, $ |K_{R^i}| $ and $ |L_{R^i}| $ are homotopically equivalent for all $i \in \mathbb{N} $. We obtain the same bar code representation for filtrations (\ref{eqFilK}) and (\ref{eqFilL}). 
		
		\begin{example}\label{examGraphAcyclic}
			Let $ R_1 $ be a self-relation on $X$ given by the graph in the Figure \ref{figGraphExamFil}(a). $R_1$ has 9 nodes and is acyclic. The eventual period is $ (3, 1) $. We obtain the bar code at Figure \ref{figGraphExamFil}(b). It has one generator of $ H_1 $ with the interval $ [1, 2] $ and we had $ 3 $ generators of $ H_0 $ that die early and $ 1 $ generator of $ H_0 $ that survive to infinity.
		\end{example}
		
		\begin{example}\label{examCraphCyclic}
			Let $R_2$ be a self-relation on $X$ given by the graph in the Figure \ref{figGraphExamFil}(c). $ R_2 $ has 10 nodes and is simple. The eventual period is $ (3, 4) $. We obtain the bar code at the Figure \ref{figGraphExamFil}(d). It has $ 4 $ generators of $ H_0$ that die at time $2$ and $2$ other generators that survive to infinity. 
		\end{example}
		
		\begin{figure}
  			\center
  			\subfigure[ The graph of the acyclic relation from Example \ref{examGraphAcyclic}. ]{
   				\includegraphics[height=6cm, width=6cm, scale=1.00, angle=0 ]{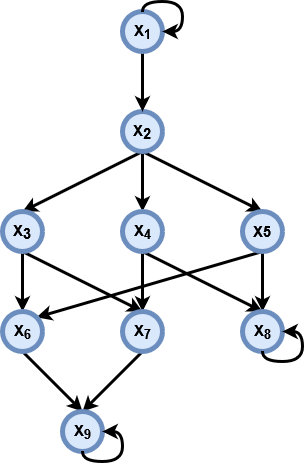}
  			}
  			\quad
  			\subfigure[ The associated bar code of the relation from Example \ref{examGraphAcyclic}. We use the filtration with $ K_{R_1^i} $. The first bar in orange is a generator in $H_1$ and the others four bars in blue are generators in $ H_0 $. ]{
   				\includegraphics[height=6cm, width=7cm, scale=1.00, angle=0 ]{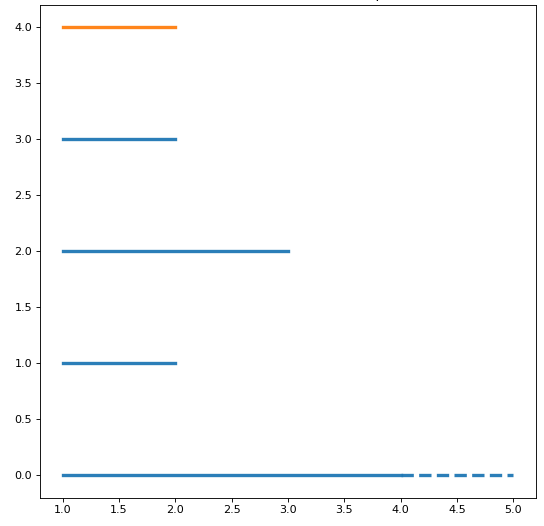}
  			}
  			\quad
  			\subfigure[ The graph of the relation with multiple cycles from Example \ref{examCraphCyclic}.]{
   				\includegraphics[height=6cm, width=6cm, scale=1.00, angle=0 ]{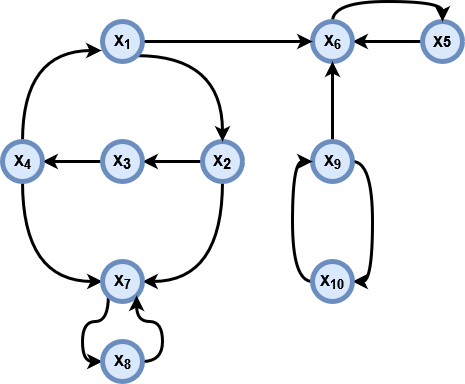}
            }
              			\subfigure[ The associated bar code of the relation from Example \ref{examCraphCyclic}. We use the filtration with $ K_{R_2^i} $. The six bars in blue are generators in $ H_0 $. ]{
   				\includegraphics[height=6cm, width=7cm, scale=1.00, angle=0 ]{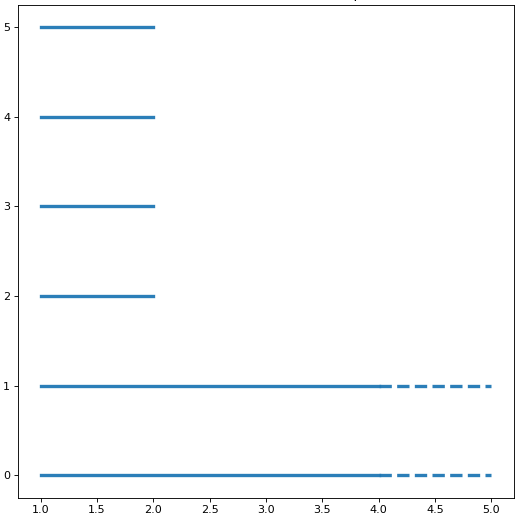}
  			}
  			\quad
            \caption{In these figures, we have the graph on the left and the associated bar code diagram on the right for Examples \ref{examGraphAcyclic} and \ref{examCraphCyclic}. Dashed lines in bar code mean it goes to infinity. }
            \label{figGraphExamFil}
		\end{figure}				
			
			For the next results, we compute the $0$th homology of the Dowker complexes for different types of relation. 			
				
			If $ R $ is a finite acyclic self-relation, then it has an eventual period $p = 1$ and $ j \in \mathbb{N}_{>0} $ such as $ R^j = R^i $ for all $i \geq j $. So we denote this relation $ R^j $ by $ R^{\infty} $, because it converges to a relation when $ i \to \infty $.	
				
		\begin{definition}\label{dfnMinMaxRel}
			We say that $ x \in X $ is a minimum for a self-relation $ R $, if there exists no $ y \in X $ such as $ x \neq y $ and $ x R^{\infty} y $. We denote the set $ U_x := \{ y \in X \mid y R^{\infty} x \} $.
			
			We say that $ x \in X $ is a maximum for a self-relation $ R $, if there exists no $ y \in X $ such as $ x \neq y $ and $ y R^{\infty} x $. We denote the set $ D_x := \{ y \in X \mid x R^{\infty} y \}  $.
		\end{definition}		

			The maximums and minimums of an acyclic relation are important, because they are responsible for the maximal simplices of $ K_{R^{\infty}} $ and $ L_{R^{\infty}} $.

		\begin{lemma}\label{lemMaxSimp}
			Let $ R $ be a finite acyclic self-relation on $X$ with $ \Dom R = X $. Then, the maximal simplices of $ K_{R^{\infty}} $ are given by the minimums of $ R^{\infty} $.
 		\end{lemma}
		\begin{proof}
			We have that $ \Dom R = X $ it implies that $ \Dom R^{\infty} = X $. Then, for all $ x \in X $ there exists $ z \in X $ such that $ x R^{\infty} z $ and $ z $ is a minimum. Let $ \sigma_y = [x_1, x_2, \ldots, x_n] \in K_{R^{\infty}} $ be an arbitrary simplex. We have an $ y \in X $ such as $ x_i R^{\infty} y $ for all $ i =1,2, \ldots, n $. By the first argument, there exists a minimum $ z \in X $ such as $ y R^{\infty} z$. Therefore, we have $ x_i R^{\infty} y R^{\infty} z $. Then, $ \sigma_y \subset \sigma_z $.
		\end{proof}		
		
		We can do a similar result with $ L_{R^\infty} $ by using the maximums of $R^{\infty}$, if $ \Ima R = X $.			
		
		\begin{theorem}\label{corGradDowker}
			Let $ R  $ an acyclic finite self-relation on $X$ with $\Dom R = X $. 
			\begin{equation} \label{eqGradDowker}
				 \text{ number of connected components of } G_R  = \dim H_0(K_{R^{\infty}}) = \dim H_0(L_{R^{\infty}})
			\end{equation}						
			
		\end{theorem}
		\begin{proof}
			First, by Dowker's Theorem, we have $ \dim H_0(K_{R^{\infty}}) = \dim H_0(L_{R^{\infty}}) $. We suppose that $ G_R $ is connected and show that $ \dim H_0(K_{R^{\infty}}) = 1 $. More precisely, we show that $ K_{R^{\infty}} $ is edge-connected. We have that $ \Dom R = X $ implies that for all $ x \in X $,$[ x ] \in K_{R^{\infty}} $.
			
			 Let $ x, x' \in X $. There exists $ y_1 \in X $ a minimum such as $ x R^{\infty} y_1 $ and $ y_1 R^{\infty} y_1 $. This implies that $ e_1 = [x, y_1] \in K_{R^{\infty}} $. We also have that there exists $ y_n \in X $ a minimum such as $ x' R^{\infty} y_n $ and $ y_n R^{\infty} y_n $. This implies that $ e_n = [x', y_n] \in K_{R^{\infty}} $.
			
			Since $ G_R $ is connected, there exists a $(y_1, y_n)$-path of finite length. We denote this sequence by $ y_1, z_1, z_2, z_3,$ $\ldots, z_m, y_n $. Without loss of generality, we take the shortest path. There exist a $ i $ such as $ z_i \in U_{y_1}$ and $ z_{i+1} \notin U_{y_1} $. First, we have $ e_2 = [y_1, z_i] \in K_{R^{\infty}} $ and $ z_i R z_{i+1} $. There exists $ y_2 \neq y_1 $ a minimum such as $ z_{i+1} R^{\infty} y_2 $. This implies that $ z_i R^{\infty} y_2 $ and we have the edge $ e_3 = [z_i, y_2] \in K_{R^{\infty}} $. We can repeat this process with the $ (y_2, y_n) $-path until we obtain a sequence of edges that connect the vertex $ [x] $ and $ [x'] $. We obtain that $ K_{R^{\infty}} $ is edge-connected.
			
			Now suppose that $ G_R $ is not connected. Let $ H $ be a connected component of $ G_R $. Then, for all $ x \in H $ and for all $ y \notin H $, we have that $ x \notin R(y) $ and $ y \notin R(x)$. It implies that for each connected component gives a single generator for $ H_0(K_{R^{\infty}}) $. 

	We can construct a map $ j : cc(G_R) \to H_0(K_{R^{\infty}}) $ that sends the connected components of $ G_R $ to the generators of $H_0(K_{R^{\infty}})$. By the previous argument, we can make this map $j$ injective . For any generator $g$ in $ H_0(K_{R^{\infty}}) $, there exists a $ x \in X $ such as $g$ is homologous to $[x]$ because $ \Dom R = X $. This implies there exists a $ H \in cc(G_R) $ such as $ x \in H $. We obtain that the map $ j $ is bijective and the equality (\ref{eqGradDowker}).

		\end{proof}
		
		We can show a similar proof for simple relations.
		
		\begin{theorem}\label{thDimH0simpleRel}
			Let $ R $ be a finite simple self-relation on $X$ with $ \Dom R = X $ and $(j, p)$ be an eventual period. Assume that $ G_R $ is connected. There exists a $r \in \mathbb{N}$ such that :
			\begin{equation}
				\dim H_0(K_{R^j}) = \dim H_0(L_{R^j}) = \text{ number of connected components of } G_{R^r}.
			\end{equation}
		\end{theorem}
		\begin{proof}
			We can find $ q $ big enough so that $ R^q $ is acyclic, because $ R $ is a simple relation. We choose a $ i \in \mathbb{N} $ such as $ iq > j $. We fix $ r = iq $. We also have that $ R^{r} $ is also acyclic. By Corollary \ref{corEvePerRel}, we have that $ K_{R^j} = K_{R^r} $. By Theorem \ref{corGradDowker}, we know that $ \dim H_0(K_{R^j}) = \text{ number of connected components of } R^{r} $.  
		\end{proof}				
		
		If $ G_R $ has more than one connected component, we apply this theorem for each connected component of $ G_R$ by using similar arguments as the proof of Theorem \ref{corGradDowker}. From preceding results, if $ \Dom  R \neq X $ but $ \Ima R = X $, we can redo the proofs with $ L_{R^j} $. Another approach is to use $ R^{-1} $, because $ \Dom R^{-1} = X $.

		\begin{remark}
			In Example \ref{examCraphCyclic}, it is a simple relation. We have $ R_2^4 $ is acyclic. The graph of $ R_2^4 $ is shown in Figure \ref{figgraphR4}. It has two connected components and the bar code from Figure \ref{figGraphExamFil}(b) has $2$ bars goes to infinity. It is expected from Theorem \ref{thDimH0simpleRel}.
		\end{remark}		
		
		\begin{figure}
			\center
   			\includegraphics[height=5cm, width=7cm, scale=1.00, angle=0 ]{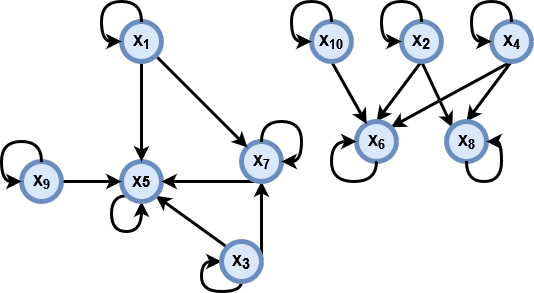}
 			\caption{The graph of $ R_2^4 $ from Example \ref{examCraphCyclic}. $ G_{R_2^4} $ has 2 connected components.}
 			\label{figgraphR4}
		\end{figure}

		We have shown earlier in Corollary \ref{corShiEqui} and Proposition \ref{propBoolMatJ} that a strongly connected self-relation $R$ is shift equivalent to a matrix $J$, if $tr(R) > 0 $. But we will like to have a result for any strongly connected relations. But, first we need some definition and other results from other papers. Let $ gcd(a, b) $ be the great common divisor of $a$ and $b$. We define :
		
		\begin{equation}\label{eqGCDq}
			q = gcd(n_1, n_2, n_3, \ldots) 
		\end{equation}
		where $ n_i $ is the length of a cycle and $ i \in \mathbb{I} $ is the set of all different cycles from $R$.
		
		We obtain this proposition :
		\begin{proposition}[Proposition 6.12 in \cite{arSzym}] \label{propQdivideJ}
				Let $ R $ be a strongly connected self-relation on $X$, $(j, p)$ be the eventual period and $ q $ defined by (\ref{eqGCDq}). We have $ q | j $.
		\end{proposition}
		
		Let's define a new equivalence relation $ \sim_q $ for a strongly connected self-relation $R$. We say that $ x \sim_q y  $, if for each $ (x, y)$-walk has length equal to $0$ modulo $q$. It is an equivalence relation.
		
		\begin{proposition}[Proposition 6.16 in \cite{arSzym}]	
			Let $ R $ be a strongly connected self-relation on $X$. Let $ q $ defined as (\ref{eqGCDq}). Then, $ \sim_q $ is an equivalence relation in $X$ with exactly $q$ distinct equivalence classes.
		\end{proposition}	
			
		We need one more Lemma before showing our final result.
		
		\begin{lemma}[Lemma 6.25 in \cite{arSzym}]\label{lem625}
			Let $ R $ be a strongly connected self-relation on $X$ and $ (j, p) $ an eventual period. Then,
			\begin{equation}
				x \sim_q x' \implies R^j(x) = R^j(x').
			\end{equation}			 
		\end{lemma}
		
		We are going to show that the number of class equivalence of $ \sim_q $ is equal to $ \dim(H_0(K_{R^j})) $ for a strongly connected self-relation with eventual period $(j, p)$.
		
		\begin{theorem}
			Let $ R $ be a finite self-relation on $X$ with an eventual period $(j, p)$, $ R $ is strongly connected, $ q $ defined by (\ref{eqGCDq}). Then, we have :
				
			\begin{equation}\label{eqThStrCon}
				\text{number of } [x]_{\sim_q} = q = \dim(H_0(K_{R^j})) = \dim(H_0(L_{R^j})).
			\end{equation}
		\end{theorem}				
		\begin{proof}
			We show that for any $ x, y \in X$, if  $ x \sim_q y  $, then $x$ and $y$ are edge-connected and if $ x \not\sim_q y $, then $ x $ and $y $ are not edge-connected.
			
		First, we suppose that $ x \sim_q y $. By Lemma \ref{lem625}, we have that $ R^j(x) = R^j(y) \neq \emptyset $. There exists a $ z \in R^j(x) $. It implies that $ [x, z] $ and $ [y, z] $ are in $ K_{R^j}$. So, each vertex in the same equivalence class is edge-connected. 
		
		Now, we suppose that, $ x \not\sim_q y $. There exists a $(x, y) $-walk of length $n$ modulo $ q $ where $ n \neq 0 $. Let show that $R^j(x) \cap R^j(y) = \emptyset$ . Let's suppose there exists a $ z \in R^j(x) \cap R^j(y) $. This implies there exists a $(x, z)$-walk of length $ j $ and a $(y, z) $-walk of length $j$. But, from Proposition \ref{propQdivideJ}, $ q | j $. This implies $ x \sim_q z $ and $ y \sim_q z $. But $ \sim_q $ is an equivalence relation. We obtain that $ x \sim_q y $ which is a contradiction. We obtain that $ x \not\sim _q y $ implies $ R^j(x) \cap R^j(y) = \emptyset $. We obtain that if $ x \sim_q y $ then they are edge-connected. But, if $ x \not\sim_q  y $, then they are not edge-connected. There is $q$ different equivalence classes. The proof is complete.
		\end{proof}
		
		For the case of $ R $ is an arbitrary relation, it is harder to find its homology $H_0$. Also, for higher dimensions of the homology groups, it's hard to tell what happens. Further investigations are needed for both cases.

		Now, we return to the filtrations defined earlier. There are two other types of filtrations that we can use. If $R$ is a self-relation on $X$ with $ \Dom R = X = \Ima R$, then we can use both filtrations. But we obtain the same bar codes for both. That holds because, for each $ i \in \mathbb{N}$, $ |K_{R^i}| $ is homotopically equivalent to $ |L_{R^i}| $ for any self-relation. We might need to come with other types of filtration. We suggest two other types of filtration. 
		
		It will be interesting to use a zigzag filtration \cite{arTdaZigZag} with $ K_{R^i} $ and $ L_{R^i} $ by alternating them. It will probably depend on the relation. Further investigations are needed. 
		
		We will present an interesting bi-filtration with $ K_{R^m} $ and $ L_{R^n} $. We have that, if $ K_{R_{m}} \subset K_{R_{m + 1}} $ and $ L_{R^n} \subset L_{R^{n+1}} $, then $ K_{R^m} \cap L_{R^n} \subset K_{R^{m+1}} \cap L_{R^{n}}  $ and $ K_{R^m} \cap L_{R^n} \subset K_{R^m} \cap L_{R^{n+1}} $ for all $ m, n \in \mathbb{N} $. We obtain this bi-filtration :

		\begin{center}
		% Start tikz
			\begin{tikzcd}
				& \vdots \arrow[hookrightarrow]{d} & \vdots \arrow[hookrightarrow]{d} & \\
				\ldots \arrow[hookrightarrow]{r} 
				& K_{R^m} \cap L_{R^n} \arrow[hookrightarrow]{r}  \arrow[hookrightarrow]{d} 
				& K_{R^{m+1}} \cap L_{R^n} \arrow[hookrightarrow]{r} \arrow[hookrightarrow]{d}
				& \ldots  \\
				\ldots \arrow[hookrightarrow]{r} & K_{R^m} \cap L_{R^{n+1}} \arrow[hookrightarrow]{r} \arrow[hookrightarrow]{d}
				& K_{R^{m+1}} \cap L_{R^{n+1}} \arrow[hookrightarrow]{r} \arrow[hookrightarrow]{d}
				& \ldots \\
				& \vdots & \vdots  &
			\end{tikzcd}
					
		\end{center}
		% End Tikz		 
		 
		The computation of the bi-filtration is also finite. Because the relation $ R $ is finite and $ \Dom R = X = \Ima R $. We obtain an eventual period $(j, p)$. In the bi-filtration, there are, at maximum, $ j^2 $ different simplicial complexes to compute.
		
		One may ask why the intersection is a good idea to consider. Let's explain it in more details. Let $ R $ be a self-relation on $X$ with $ \Dom R = X = \Ima R $ and $ m, n \in \mathbb{N} $. Let $ \sigma \in K_{R^m} \cap L_{R^n} $ where $ \sigma = [x_1, x_2, \ldots, x_d] $.  Then, there exists $ x_{\omega} \in X $ such that $ x_i R^{m} x_{\omega} $ and there exists $ x_{\alpha} \in X $ such that $ x_{\alpha} R^n x_i$ for all $ i $. Another way to see this is, for each $ x_i \in \sigma $, there exists a $(x_{\omega}, x_{\alpha})$-walk of length $ m + n $ passing through $ x_i $. We can subdivide this $( x_\omega, x_\alpha)$-walk into a $(x_\omega, x_i) $-walk of length $m$ and a $(x_i, x_\alpha)$-walk of length $n$. So, by only the existence of a simplex $ \sigma $ in $ K_{R^m} \cap L_{R^n} $, each vertex of $ \sigma $, it has a walk with a common starting point and a common ending point of same length going through the vertex. It will be interesting to study these Dowker complexes when $ m $ tends to infinity, $n$ tends to infinity or both.

	\section{Conclusion}
		
		In summary, we used the Dowker complexes to study some properties of self-relation. First, we defined the right morphism and left morphism. We also generalized it to the case of multivalued maps called multi-right morphism and multi-left morphism. The existence of a right or multi-right morphism between $R_1$ and $R_2$ implies that the $K_{R_1}$ is included in $ K_{R_2} $. Similarly, the existence of a left or a multi-left morphism between $ R_1 $ and $ R_2 $ implies that $ L_{R_1} $ is included in $ L_{R_2} $. We have shown that two relations are conjugate implies that they have homotopically equivalent Dowker complexes. We have also shown that if two relations are shift equivalent, then their Dowker complexes are homotopically equivalent at some power of the relations. We were interested in self-relation which is equivalent of a type of directed graph. We have obtained two nice properties. If $R$ is finite and $ \Dom R = X = \Ima R $, then we have that $ K_{R^i} \leq K_{R^{i+1}} $ and $ L_{R^i} \leq L_{R^{i+1}} $ for all $i \in \mathbb{N}$. Moreover, there exists a $ j \in \mathbb{N} $ such as $ K_{R^j} = K_{R^{i+j}} $ and $ L_{R^j} = L_{R^{i+j}} $ for all $ i \in \mathbb{N} $. With these two properties, we defined two filtrations with $ K_{R^i} \hookrightarrow K_{R^{i + 1}} $ and $ L_{R^i} \hookrightarrow L_{R^{i + 1}} $. Also, the filtration ends at some finite time. Finally, we proved some results about the $0$th homology for some types of self-relation at some power. We also proposed the intersection filtration and the zigzag filtration.
		
		We have put some foundations to study directed graph using Dowker complexes. Moreover, we think it might be a useful tool to study the dynamics of finite data define by a directed graph of Dowker complexes. Let $ R $ be a self-relation with an eventual period $(j, p) $. The positive or forward invariant is given by the existence of a simplex in $ K_{R^j} $ when $ j $ converges to infinity. The negative or backward invariant is given by the existence of a simplex in $ L_R^j $ when $j $ converges to infinity. But, by the stabilization of Dowker complexes, we can compute it in finite time. Finally, if we want to study the invariant of $R$, this is given by the existence of a simplex in $ K_{R^j} \cap L_{R^j} $ which is the intersection of the forward and the backward invariant. But further investigations are needed. The idea is to use the structure of the Dowker complexes to encode the dynamics of finite data.

%------------------------------------------------------------------------------
%  BIBLIOGRAPHY
%-------------------------------------------------------------------------------

	%\addcontentsline{toc}{chapter}{BIBLIOGRAPHIE}
	\clearpage
	\typeout{}
	\bibliography{biblio}

\end{document}